\documentclass[12pt]{article} 

\usepackage{amsmath}
\usepackage{amsthm}
\usepackage{amsfonts}
\usepackage{mathrsfs}
\usepackage{stmaryrd}
\usepackage{setspace}
\usepackage{fullpage}
\usepackage{amssymb}
\usepackage{breqn}
\usepackage{enumitem}
\usepackage{bbold} 
\usepackage{authblk}
\usepackage{comment}
\usepackage{hyperref}
\usepackage{pgf,tikz}
\usepackage{graphicx}
\usepackage{subcaption}

\newtheorem*{thm*}{Theorem}
\newtheorem{thm}{Theorem}[section]

\newtheorem{lemma}[thm]{Lemma}

\newtheorem*{prop*}{Proposition}

\newtheorem{clm}[thm]{Claim}

\newcommand\ex{\ensuremath{\mathrm{ex}}}
\newcommand\biex{\ensuremath{\mathrm{biex}}}
\newcommand\og{\ensuremath{\mathrm{og}}}

\newcommand\cD{{\mathcal D}}

\newcommand\cG{{\mathcal G}}

\newcommand\cN{{\mathcal N}}

\newcommand\cT{{\mathcal T}}

\newcommand{\ignore}[1]{}

\title{Some exact results for non-degenerate generalized Turán problems}
\author{Dániel Gerbner\\\small Alfr\'ed R\'enyi Institute of Mathematics\\
\small \texttt{gerbner.daniel@renyi.hu}}

\date{}

\begin{document}

\maketitle

\begin{abstract}
The generalized Tur\'an number $\mathrm{ex}(n,H,F)$ is the maximum number of copies of $H$ in $n$-vertex $F$-free graphs. We consider the case where $\chi(H)<\chi(F)$. There are several exact results on $\ex(n,H,F)$ when the extremal graph is a complete $(\chi(F)-1)$-partite graph. We obtain multiple exact results with other kinds of extremal graphs.
\end{abstract}




\section{Introduction}

One of the most studied questions in extremal Combinatorics is the following: what is the largest number $\ex(n,F)$ of edges that an $n$-vertex graph can have, if it does not contain $F$ as a subgraph? Tur\'an \cite{T} showed that in the case $F=K_{r+1}$, the largest number of edges are in the Tur\'an graph $T(n,r)$, which is the complete $r$-partite graph with each part of order $\lfloor n/r\rfloor$ or $\lceil n/r\rceil$. Erd\H os, Stone and Simonovits \cite{ES1966,ES1946} showed that if $\chi(F)=r+1>2$, then $\ex(n,F)=(1+o(1))|E(T(n,r))|$.

A straightforward generalization of the above question is when instead of the number of edges, we consider the number of subgraphs isomorphic to a given graph $H$. Let $\cN(H,G)$ denote the number of copies of $H$ in $G$. Given graphs $H$ and $F$ and a positive integer $n$, their generalized Tur\'an number is $\ex(n,H,F)=\max\{\cN(H,G): \text{ $G$ is an $n$-vertex $F$-free graph}\}$. The systematic study of these numbers was initiated by Alon and Shikhelman \cite{ALS2016}.

One particular topic that has attracted lot of attention is the study of when the Tur\'an graph contains the most copies of $H$ among $F$-free graphs. More generally, when does a complete $(\chi(F)-1)$-partite graph contains the most copies of $H$? We say that $H$ is \textit{$F$-Tur\'an-good} if $\chi(H)<\chi(F)$ and $\ex(n,H,F)=\cN(H,T(n,r))$ for sufficiently large $n$. We say that $H$ is \textit{weakly $F$-Tur\'an-good} if $\chi(H)<\chi(F)$ and for sufficiently large $n$, $\ex(n,H,F)=\cN(H,T)$ for some complete $(\chi(F)-1)$-partite graph $T$. Note that for a given graph $H$, a straightforward but complicated computation determines which $n$-vertex complete $r$-partite graph contains the most copies of $H$. The very first result in the area is due to Zykov \cite{zykov} and states that if $k<r$, then $K_k$ is $K_r$-Tur\'an-good. The systematic study of $F$-Tur\'an-good graphs was initiated by Gy\H ori, Pach and Simonovits \cite{gypl} in the case $F=K_r$ and by Gerbner and Palmer \cite{gp3} for general $F$. 

A phenomenon often seen in this area is the so-called \textit{stability}. It refers to the property that if some $F$-free graph contains almost $\ex(n,F)$ edges (almost $\ex(n,H,F)$ copies of $H$), then it is in some sense close to the extremal graph. There are different versions of stability based on what ''almost'' and ''close'' means. The most studied version is based on the following notion. Given two graphs $H$ and $G$ on the same vertex set, the \textit{edit distance} of $H$ and $G$ is the least number of edges that need to be added to $H$ and removed from $H$ to obtain $G$. Slightly imprecisely, when we say that $H$ has edit distance at most $k$ from $T(n,r)$, we mean that we can obtain a graph isomorphic to $T(n,r)$ on the vertex set $V(G)$ by adding and deleting at most $k$ edges.

The classical Erd\H os-Simonovits stability theorem \cite{erd1,erd2,sim} states that if an $n$-vertex $F$-free graph $G$ has $\ex(n,F)-o(n^2)$ edges, then the edit distance of $G$ from $T(n,\chi(F)-1)$ is $o(n^2)$. We say that $H$ is \textit{$F$-Tur\'an-stable} if for every $n$-vertex $F$-free graph $G$ that contains $\ex(n,H,F)-o(n^{|V(H)|})$ copies of $H$, the edit distance of $G$ from $T(n,\chi(F)-1)$ is $o(n^2)$. We say that $H$ is \textit{weakly $F$-Tur\'an-stable} if for every $n$-vertex $F$-free graph $G$ that contains $\ex(n,H,F)-o(n^{|V(H)|})$ copies of $H$, the edit distance of $G$ is $o(n^2)$ from some complete $(\chi(F)-1)$-partite graph $T$. Using this language, the Erd\H os-Simonovits stability theorem states that $K_2$ is $F$-Tur\'an-good for every non-bipartite graph $F$.


The first stability result in generalized Tur\'an problems is due to Ma and Qiu \cite{mq}. They showed that 
if $\chi(F)=r+1$ and $k\le r$, then $K_k$ is $F$-Tur\'an-stable. They used this result to show that if $F$ has a color-critical edge, then $K_k$ is $F$-Tur\'an-good.
More precisely, they proved a bound for every $F$ that happens to give an exact result if $F$ has a color-critical edge. The \textit{decomposition family} $\cD(F)$ of $F$ consists of every bipartite graph that can be obtained by deleting $r-1$ classes from an $(r+1)$-coloring of $F$. Let $\biex(n,F)$ denote the maximum number of edges in an $n$-vertex graph that does not contain any member of the decomposition family of $F$. In particular, if $F$ has a color-critical edge, then $K_2$ is in the decomposition family of $F$, thus $\biex(n,F)=0$.

\begin{thm}[Ma, Qiu \cite{mq}]\label{maqi}
For any $k\le r$, $\ex(n,K_k,F)=\cN(K_k,T(n,r))+\biex(n,F)\Theta(n^{k-2})$.
\end{thm}

This approach where we prove that $H$ is $F$-Tur\'an-stable and use it to prove that $H$ is $F$-Tur\'an-good can also be found in \cite{lm,murnir,hhl,ger2}. Finally, Gerbner \cite{gerb2} provided the following general formulation.

\begin{thm}\label{haszn}
Let $\chi(F)>\chi(H)$ and assume that $F$ has a color-critical edge. If $H$ is weakly $F$-Tur\'an-stable, then $H$ is weakly $F$-Tur\'an-good. 
\end{thm}

Let us remark that the same property is also implied by the weaker assumption that there is an $n$-vertex $F$-free graph $G$ with $\cN(H,G)= \ex(n,H,F)$ that has edit distance $o(n^2)$ from a complete $(\chi(F)-1)$-partite graph. We emphasize that what makes the above theorem especially useful is the simple observation from \cite{gerb2} that if $H$ is weakly $K_{\chi(F)}$-Tur\'an-stable, then $H$ is weakly $F$-Tur\'an-stable, thus one stability result gives infinitely many exact results.

In this paper we extend this result by going beyond $(\chi(F)-1)$-partite graphs: we obtain exact results when the extremal construction is obtained by adding further edges to a complete $(\chi(F)-1)$-partite graph. One such result has been obtained by Gerbner and Patk\'os \cite{gerpat}. Let $B_{r,1}$ denote the graph consisting of two copies of $K_r$ sharing a vertex. It was shown in \cite{cgpw} that among $B_{r,1}$-free graphs, the most edges are contained in the Tur\'an graph plus an additional edge. It was extended in \cite{gerpat} to any $K_r$ with $r<k$ in place of $K_2$. There are further results when $B_{3,1}$ is forbidden in \cite{ger2}.



Our first result is a common generalization of Theorems \ref{maqi} and \ref{haszn}.

\begin{thm}\label{thm1}
\textbf{(i)} Let $r+1=\chi(F)>\chi(H)$ and assume that $H$ is weakly $F$-Tur\'an-stable. Then $\ex(n,H,F)\le\cN(H,T)+\biex(n,F)\Theta(n^{|V(H)|-2})$ for some $n$-vertex complete $r$-partite graph $T$. 
Moreover, for every $n$-vertex $F$-free graph $G$ with $\cN(H,G)=\ex(n,H,F)$ there is an $r$-partition of $V(G)$ to $A_1,\dots,A_r$, a constant $K=K(F)$ and a set $B$ of at most $rK(\sigma(F)-1)$ vertices such that each member of $\cD(F)$ inside a part shares at least two vertices with $B$, every vertex of $B$ is adjacent to $\Omega(n)$ vertices in each part, and every vertex of $A_i\setminus B$ is adjacent to $o(n)$ vertices in $A_i$ and all but $o(n)$ vertices in $A_j$ with $j\neq i$.

\smallskip

\noindent\textbf{(ii)} $\ex(n,H,F)=\cN(H,T)+\biex(n,F)\Theta(n^{|V(H)|-2})$ if there is a connected component $H'$ of $H$ such that one of the followings hold.

\textbf{(a)} $H'$ has a coloring with at most $r$ colors that is almost proper: there is exactly one edge whose endpoints have the same color.

\textbf{(b)} $\biex(n,H')=o(\biex(n,F))$.
\end{thm} 

In particular, if $F$ has a color-critical vertex, there are no members of $\cD(F)$ inside the parts $A_i$. 
An example where \textbf{(ii)} does not hold is $\ex(n,C_4,F_2)$. It was shown in \cite{gp3} that $C_4$ is $F_2$-Tur\'an-good, i.e., $\ex(n,C_4,F_2)=\max\{\cN(H,T)\}$, while $\biex(n,F)=1$. A more general result is in \cite{gerbner2}, showing that for any $F$ with a color-critical vertex there is a graph $H$ that is $F$-Tur\'an-good. Again, the error term is 0 instead of $\biex(n,F)\Theta(n^{|V(H)|-2})$.


We can obtain several exact results. First we generalize a theorem of Moon \cite{moon} that determines $\ex(n,F)$ where $F$ consists of $s$ vertex-disjoint copies of $K_{r+1}$.

\begin{thm}\label{disconn}
Let $F$ consist of $s>1$ components with chromatic number $r+1$, each with a color-critical edge, and any number of components with chromatic number at most $r$. Let $H$ be a weakly $F$-Tur\'an-stable graph and $n$ sufficiently large. Then $\ex(n,H,F)=\cN(H,T)$ for a complete $(s+r-1)$-partite graph $T$ with $s-1$ parts of order 1.
\end{thm}


From now on we will focus on the case where $F$ has a color-critical vertex. An example where we can obtain an exact result is when all the parts can contain any $\cD(F)$-free graphs at the same time. Note that this does not necessarily mean that we embed a $\cD(F)$-free graph with the maximum number of edges into the parts.

Let us consider the complete $(r+1)$-partite graph $K_{1,a,\dots,a}$. Clearly its decomposition family contains the star $K_{1,a}$, thus graphs with maximum degree at most $a-1$ can be embedded into each part. Let us call a graph \textit{almost $a$-regular} if it is either $a$-regular or has one vertex of degree $a-1$ and each other vertex has degree $a$.

Let $\cT_0^{(s)}(n,r)$ denote the following family of graphs. We take a complete $r$-partite graph $T$, and for each part $A_i$, we embed an almost $s$-regular graph. It is easy to see that these graphs are $K_{1,a,\dots,a}$-free. Let $\cT^{(s)}(n,r)$ denote the subfamily of $\cT_0^{(s)}(n,r)$ where $T$ is the Tur\'an graph $T(n,r)$. Simonovits \cite{sim} showed that the graphs in $\cT^{(a-1)}(n,r)$ have the most edges among $K_{1,a,\dots,a}$-free graphs. 

\begin{thm}\label{treecompl}
Let $F$ be the complete $(r+1)$-partite graph $K_{1,a,\dots,a}$, $H$ be a weakly $F$-Tur\'an-stable graph and $n$ sufficiently large. 

\textbf{(i)} $\ex(n,H,F)=\cN(H,T)$ for some $T\in \cT_0^{(a-1)}(n,r)$.

\textbf{(ii)} If $H$ is a forest, then $\ex(n,H,F)=\cN(H,T)$ for every $T\in \cT_0^{(a-1)}(n,r)$ where the graph embedded into each part has girth at least $|V(H)|$.

\end{thm}

We remark that \textbf{(i)} is not immediate from the preceding argument, as it is possible that a graph embedded into a part with less edges contains more copies of some (subgraph of) $H$. Consider now cliques $K_k$. It is not hard to show that we need to embed into each part of a complete $r$-partite graph an almost $a$-regular graph with the maximum number of triangles, and among those we need a graph with the maximum number of $K_4$s, and so on. However, to determine these graphs does not seem to be a simple problem. A very special case is solved in \cite{zcggyh}.

We also remark that linear forests, (thus paths) are  $F$-Tur\'an-stable for every $r$, and forests $H$ containing a matching of size $\lfloor |V(H)|/2\rfloor$ are $F$-Tur\'an-stable for $r=3$. Thus, \textbf{(ii)} gives an exact result for those graphs.

 
Assume now that $\biex(n,F)=O(1)$. This will be useful since we add to a complete $r$-partite graph $T$ and delete from $T$ $O(1)$ edges, thus it is enough to obtain asymptotic results on the number of copies of $H$ containing those edges. 

 \begin{thm}\label{thm3}
 Let $r+1=\chi(F)>\chi(H)$, assume that $H$ is $F$-Tur\'an-stable and $F$-Tur\'an-good and $n$ is sufficiently large. Assume that $H$ is a forest
  and $\biex(n,F)=O(1)$. Then $\ex(n,H,F)=\cN(H,G)$ for some $G$ 
  with $\ex(n,F)=|E(G)|$.
 \end{thm}

Note that $H$ being $F$-Tur\'an-stable implies that $H$ is weakly $F$-Tur\'an-good, and the extremal complete $r$-partite graph $T$ is almost balanced, but $T$ is not necessarily the Tur\'an graph.

Given $F$ with chromatic number $r+1$, let us call an $F$-free $n$-vertex graph $G$ \textit{nice} for $F$ if $G$ contains $T(n,r)$ and has $\ex(n,F)$ edges.

\begin{thm}\label{bipa} Let $\chi(F)=3$, $\chi(H)=2$, assume that $H$ is $F$-Tur\'an-stable and $F$-Tur\'an-good, $\biex(n,F)=O(1)$ and $n$ is sufficiently large. If there is a nice graph for $F$, then there is a nice graph $G_0$ for $F$ with $\ex(n,H,F)=\cN(H,G_0)$.
\end{thm}


An even simpler special case is when $\biex(n,F)=1$, i.e., when $\cD(F)$ contains both the star and the matching with two edges. The main example is $B_{r,1}$. Another example is the following graph $Q_r$. Let $r\ge 3$. We take $K_{1,2,a_3,\dots,a_{r+1}}$ with parts $A_i$ of order $a_i\ge 2$, and remove all but two independent edges between $A_r$ and $A_{r+1}$. 
Let $\cG_m$ denote the family of graphs obtained the following way. We take a complete $r$-partite graph $T$, add one edge $u_iv_i$ into each of $m$ parts, and delete the edges $u_iu_j$ and $v_iv_j$. Let $G_m$ denote the element of $\cG_m$ where $T=T(n,r)$ and the $m$ additional edges are placed into $m$ smallest parts.

\begin{thm}\label{fan}

\textbf{(i)} Let $H$ be weakly $Q_r$-Tur\'an-stable. Then $\ex(n,H,Q_r)=\cN(H,G)$ for some $G\in \cG_m$, $m\le r$. In particular, if  $q:=(r-1)/2(r-k+1)$ is not an integer, then $\ex(n,K_k,Q_r)=\cN(K_k,G_{\lceil q\rceil})$. If $q$ is an integer, then either $\ex(n,K_k,Q_r)=\cN(K_k,G_{q})$ or $\ex(n,K_k,Q_r)=\cN(K_k,G_{q+1})$.

\textbf{(ii)} Let $H$ be $B_{r,1}$-Tur\'an-stable. Then $\ex(n,H,B_{r,1})=\cN(H,G')$, for some $G'$ that is obtained from a complete $r$-partite graph by adding an edge into one of the parts.
\end{thm}

The first statement of the above theorem gives an example where the extremal graph is not the same for $\ex(n,H,F)$ as for $\ex(n,F)$, even though both contain the Tur\'an graph. The second statement generalizes most of the known exact results when $\chi(H)<\chi(F)$ and $F$ does not have a color-critical edge.

Finally, we show two simple ways to obtain weakly $F$-Tur\'an-stable graphs. The \textit{odd girth} $\og(G)$ of a graph $G$ is the shortest odd cycle in it. Its connection to Tur\'an-goodness was studied in \cite{ger3}. It was shown there that $G$ is the subgraph of the $p$-blow-up of the cycle of length $\og(G)$ for some $p$, i.e., of the graph obtained the following way: we replace each vertex $v$ of $G$ with $p$ vertices $v_1,\dots,v_p$, and each edge $uv$ with $p^2$ edges $u_iv_j$, $i,j\le p$. Let $b(G)$ denote the smallest $p$ such that the above property holds.

\begin{thm}\label{stabil}
\textbf{(i)} Let $\chi(F)=r+1$ and $H$ be a weakly $F$-Tur\'an-stable graph. Assume that $H$ has a unique $r$-coloring and $H'$ is an $r$-chromatic graph obtained by adding edges but no vertices to $H$. Then $H'$ is  weakly $F$-Tur\'an-stable. 

\textbf{(ii)} Let $\chi(F)=3$ and $H$ be a weakly $F$-Tur\'an-stable graph. Let us assume that $H$ contains the $b(F)$-blow-up of $P_{\og(F)-1}$, where $u_1,\dots,u_{b(F)}$ replace the first vertex and $v_1,\dots,v_{b(F)}$ replace the last vertex of the path. Let $H'$ be the graph obtained by adding vertices $w_1,\dots,w_s,x_1,\dots x_t$ and edges $u_iw_j$, $v_ix_\ell$ for $i\le b(F)$, $j\le s$ and $\ell\le t$. Assume that $\binom{t-s}{2}\le s\le t$. Then $H'$ is $F$-Tur\'an-stable.
\end{thm}

We remark that \textbf{(i)} is a straightforward extension of a proposition from \cite{ger2}, which states the same result with weakly $F$-Tur\'an-good instead of weakly $F$-Tur\'an-stable. Observe that in \textbf{(ii)}, if $F=K_3$, then we just add leaves to the endpoints of an edge.

\section{Proofs}

We will use the following lemma from \cite{gerb2}.

\begin{lemma}\label{neww}
Let us assume that $\chi(H)<\chi(F)$ and $H$ is weakly $F$-Turán-stable, thus $\ex(n,H,F)=\cN(H,T)+o(n^{|V(H)|})$ for some complete $(\chi(F)-1)$-partite $n$-vertex graph $T$. Then every part of $T$ has order $\Omega(n)$.
\end{lemma}


We will use the following simple observation (and variations of it) multiple times.

\begin{lemma}\label{greed}
Let $G$ be an $n$-vertex graph with a partition of $V(G)$ to $A_1,\dots,A_r$ such that $|A_i|=\Theta(n)$ for each $i$. Assume that each vertex of $A_i$ is connected to all but $o(n)$ vertices outside $A_i$. Let $F'$ be an induced subgraph of $F$ such that the remaining part of $F$ is $s$-colorable. Assume that a copy of $F'$ is embedded into $G$ avoiding $A_1,\dots, A_s$. Then we can extend this copy of $F'$ to a copy of $F$ in $G$.
\end{lemma}

\begin{proof}
Let $U_1\dots, U_s$ be the color classes of the graph we obtain by deleting $F'$ from $F$. We will embed the sets $U_i$ one by one in an arbitrary order into $A_i$. When we embed $U_i$, the already embedded at most $|V(F)|$ vertices are adjacent to all the $\Theta(n)$ vertices in $A_i$ with $o(n)$ exceptions, thus we can pick the necessary vertices. 
\end{proof}

Now we are ready to prove Theorem \ref{thm1} that we restate here for convenience. 

\begin{thm*}
\textbf{(i)} Let $r+1=\chi(F)>\chi(H)$ and assume that $H$ is weakly $F$-Tur\'an-stable. Then $\ex(n,H,F)\le\cN(H,T)+\biex(n,F)\Theta(n^{|V(H)|-2})$ for some $n$-vertex complete $r$-partite graph $T$. 
Moreover, for every $n$-vertex $F$-free graph $G$ with $\cN(H,G)=\ex(n,H,F)$ there is an $r$-partition of $V(G)$ to $A_1,\dots,A_r$, a constant $K=K(F)$ and a set $B$ of at most $rK(\sigma(F)-1)$ vertices such that each member of $\cD(F)$ inside a part shares at least two vertices with $B$, every vertex of $B$ is adjacent to $\Omega(n)$ vertices in each part, and every vertex of $A_i\setminus B$ is adjacent to $o(n)$ vertices in $A_i$ and all but $o(n)$ vertices in $A_j$ with $j\neq i$.

\smallskip

\noindent\textbf{(ii)} $\ex(n,H,F)=\cN(H,T)+\biex(n,F)\Theta(n^{|V(H)|-2})$ if there is a connected component $H'$ of $H$ such that one of the followings hold.

\textbf{(a)} $H'$ has a coloring with at most $r$ colors that is almost proper: there is exactly one edge whose endpoints have the same color.

\textbf{(b)} $\biex(n,H')=o(\biex(n,F))$.
\end{thm*} 

\begin{proof}

To prove \textbf{(i)}, we will pick the numbers $\alpha,\beta,\gamma,\varepsilon>0$ in this order, such that each is sufficiently small compared to the previous one, and after that we pick $n$ that is sufficiently large. Let us consider an $n$-vertex $F$-free graph $G$ with $\ex(n,H,F)$ copies of $H$. Because of the Tur\'an-stable property, for any $\varepsilon>0$ there is a complete $r$-partite graph $T$ on $V(G)$ that can be obtained from $G$ by adding and removing at most $\varepsilon n^2$ edges. Let us pick $T$ such that we need to remove the least number of edges and let $A_1,\dots, A_r$ be the parts of $T$. Note that by the choice of $T$, each vertex in $A_i$ has at least as many neighbors in every part as in $A_i$, and we have $|A_i|\ge \alpha n$ for some $\alpha>0$ using Lemma \ref{neww}. 

Let $\beta>0$ be sufficiently small and $\gamma>0$ be sufficiently small compared to $\beta$. Let $B_i$ denote the set of vertices in $A_i$ with at least $\gamma n$ neighbors in $A_i$ and $B=\cup_{i=1}^r B_i$.

\begin{clm} There is a $K$ depending on $\gamma$ and $F$ such that
$|B|\le K(\sigma(F)-1)$.
\end{clm}

\begin{proof} This is an extension of Claim 4.2 in \cite{mq}, and the proof also extends to our case, thus we only give a sketch here. Clearly $|B|\le\varepsilon n\gamma$ by the definition of $\varepsilon$. Therefore, $v\in B_i$ has at least $\gamma n-\varepsilon n\gamma\ge \gamma n /2$ in every $A_j\setminus B_j$. 

Let $G(v)$ denote the subgraph of $G$ we get by picking $\gamma n/2$ neighbors of $v$ from every $A_j\setminus B_j$. Then at most $\varepsilon n^2$ edges are missing in $G(v)$ between parts $A_j\setminus B_j$. Therefore, the edge density in $G(v)$ is larger than $(r-2)/(r-1)$ (as that is the edge density of the $(r-1)$-partite Tur\'an graph, but we have almost all the edges of the $r$-partite Tur\'an graph). Thus, we can apply the Erd\H os-Simonovits supersaturation theorem \cite{ersi} to obtain that $G(v)$ contains at least $cn^{br}$ copies for some constant $c>0$ of the complete $r$-partite graph $K_{b,\dots,b}$ for $b=|V(F)|$.

Consider the following auxiliary bipartite graph. Part $A$ consists of the copies of $K_{b,\dots,b}$ in $\cup_{i=1}^r A_i\setminus B_i$, while the other part is $B$. A vertex $u\in A$ is adjacent to $v\in B$ if the corresponding complete $r$-partite graph is in the neighborhood of $v$. Then clearly $|A|\le n^{br}$. Each vertex of $A$ has at most $\sigma(F)-1$ neighbors in $B$, since otherwise we can find a complete $(r+1)$-partite graph in $G$ with parts of order $\sigma(F),|V(F)|,\dots, |V(F)|$, which obviously contains $F$. This implies that the number of edges in this auxiliary bipartite graph is at most $(\sigma(F)-1)n^{br}$ and at least $|B|cn^{br}$, completing the proof with $K=1/c$.
\end{proof}

Let us return to the proof of the theorem. Let $U_i$ denote the set of vertices $v$ in $A_i\setminus B_i$ such that there are at least $\beta n$ vertices in $V(G)\setminus A_i$ that are not adjacent to $v$. Clearly, $|U_i|\le \varepsilon n/\beta$.
Let $V_i=A_i\setminus (B_i\cup U_i)$. 

For each vertex $v\in U_i$ for each $i$, let us delete the edges from $v$ to vertices in $A_i$, and connect $v$ to every vertex of $V_j$ with $j\neq i$. Let $G'$ be the resulting graph. Observe that we deleted at most $\gamma n$ edges incident to $v$, thus at most $\gamma n^{|V(H)|-1}$ copies of $H$ containing $v$. On the other hand, we added at least $\beta n -r\varepsilon n/\beta -rK(\sigma(H)-1)\ge \beta n/2$ edges incident to $v$. 

We claim that these edges are in at least $\beta n^{|V(H)|-1}/2^{|V(H)|-1}$ copies of $H$. Indeed, let us fix an $r$-coloring of $H$ with color classes $W_1,\dots,W_r$. We count only the copies of $H$ where $W_j$ is embedded into $V_j$ for $j\neq i$ and $W_i$ is embedded into $V_i\cup \{v\}$. We apply the same greedy idea that we used in the proof of Lemma \ref{greed}. First we embed $W_i$: One vertex to $v$ and the other vertices into $V_i$ arbitrarily. Then we embed the other parts $W_j$ in an arbitrary order. Each time, when we want to embed $W_\ell$, the already embedded at most $|V(H)|$ vertices have at least $|V_\ell|-|V(H)|\beta n>n/2+|V(H)|$ common neighbors in $V_\ell$, as each vertex of $V_j$ has at most $\beta n$ non-neighbors in other parts. Therefore, embedding the vertices of $W_\ell$ one by one, we always have at least $n/2$ choices.

We also claim that $G'$ is $F$-free. Assume not and pick a copy $F_0$ of $F$ with the smallest number of vertices from $\cup_{i=1}^r U_i$. Clearly $F_0$ contains a vertex $v\in U_i$, as all the new edges are incident to such a vertex.  Let $Q$ be the set of vertices in $F_0$ that are adjacent to $v$ in $G'$. They are each from $\cup_{j\neq i}V_j$. Their common neighborhood in $V_i$ is of order at least $n-|V(F)|\beta n>|V(H)|$.
Therefore, at least one of them is not in $F_0$, thus we can replace $v$ with that vertex to obtain another copy of $F$ with less vertices from $\cup_{i=1}^r U_i$, a contradiction.

We obtained that $G'$ is $F$-free and contains more copies of $H$ than $G$ (a contradiction) unless $U_i$ is empty for every $i$. We also claim that there is no member of $\cD(F)$ inside $V_i$. Indeed, by Lemma \ref{greed} that would extend to a copy of $F$. Moreover, if there is a member of $\cD(F)$ that contains only one vertex $u$ from $B$, then we can restrict ourselves to the neighbors of $u$ and apply Lemma \ref{greed} to obtain a copy of $F$.

Let us count now the copies of $H$ in $G$. We have $\cN(H,T)$ copies inside $T$. There are at most $r\biex(n,F)$ edges inside $V_i$, thus there are at most $r\biex(n,F)n^{|V(H)|-2}$ copies of $H$ using some of those edges. It is left to count the copies of $H$ which contain a vertex from $B$. As $|B|\le rK(\sigma(F)-1)$, clearly there are at most $rK(\sigma(F)-1)n^{|V(H)|-1}$ such copies of $H$. If $\sigma(F)=1$, then this is 0. If $\sigma(F)>1$, then $\biex(n,F)\ge n-1$, since the star is in the decomposition family. Therefore, $rK(\sigma(F)-1)n^{|V(H)|-1}=O(\biex(n,F)n^{|V(H)|-2})$, completing the proof of \textbf{(i)}.

The lower bound in \textbf{(ii)} is given by taking $T$ and adding into one part $A_i$ of $T$ a $\cD(F)$-free graph with $\biex(|A_i|,F)$ edges. As $|A_i|=\Omega(n)$ by Lemma \ref{neww}, $\biex(|A_i|,F)=\Theta(\biex(n,F))$. Observe that in \textbf{(a)}, each new edge is in $\Theta(n^{|V(H')|-2})$ copies of $H'$. Clearly we can find $\Theta(n^{|V(H'')|})$ copies of every component $H''$ of $H$ in $T$, which completes the proof of \textbf{(a)}. 

To show \textbf{(b)}, observe that there are $\Theta(\biex(n,F))$ edges in $A_i$ that are contained in a member of $\cD(H')$. Each such copy is clearly extended a copy of $H'$ using edges of $T$. From here we can proceed as in the proof of \textbf{(a)} above.
\end{proof}

Let us continue with the proof of Theorem \ref{disconn}. Recall that it deals with $F$ consisting of $s>1$ components with chromatic number $r+1$ and a color-critical edge, and additional components with smaller chromatic number. The theorem states that if $H$ is weakly $F$-Tur\'an-stable and $n$ is large enough, then $\ex(n,H,F)=\cN(H,T)$ for some complete $(s+r-1)$-partite graph $T$ with $s-1$ parts of order 1.

\begin{proof}[Proof of Theorem \ref{disconn}] The lower bound is obvious. Let $G$ be an $n$-vertex $F$-free graph with $\ex(n,H,F)$ copies of $H$.
We will apply Theorem \ref{thm1} to obtain a partition to $A_1,\dots,A_r$ and a set $B$ with $|B|=O(1)$ such that each member of $\cD(F)$ inside parts $A_i$ contains a vertex from $B$. Assume first that there are $s$ independent edges $u_1v_1,\dots, u_sv_s$ inside the parts such that for each $i$, at least one of $u_i$ and $v_i$ are not in $B$. Observe that $u_i$ and $v_i$ have $\Omega(n)$ common neighbors in each part besides the one containing them.

Let $F_1,\dots,F_s$ denote the components of $F$ with chromatic number $r+1$. We apply the same greedy idea as in the proof of Lemma \ref{greed}. We go through the edges $u_iv_i$ one by one and extend them to $F_i$. Without loss of generality, $v_i\in A_1$. Let $B_2,\dots,B_r$ denote the parts of the graph we obtain by deleting a color-critical edge from $F_i$. We will embed the vertices in $B_j$ to $A_j$. Recall that $u_i$ and $v_i$ have $\Omega(n)$ common neighbors in $A_2$. We pick $|B_2|$ of them that avoid the vertices we already picked to be in the copy of $F$. Then the vertices we picked to be in the copy of $F_i$ have $\Omega(n)$ common neighbors in $A_3$, we pick $|B_3|$ of them that we have not picked to be in our copy of $F$, and so on. We always have to avoid $O(1)$ already picked vertices. This way we obtain an $F_i$ and ultimately $F_1,\dots,F_s$. Clearly we can pick the remaining components in a similar way to obtain $F$, a contradiction.

If $|B|\ge s$, then clearly we can find $s$ distinct vertices among their neighbors not in $B$, resulting in the contradiction. Observe that outside $B$, each vertex is in $o(n)$ edges inside parts. It is easy to see that the edges inside parts but outside $B$ form at most $s-1-|B|$ stars plus $O(1)$ further edges. Therefore, deleting all the edges inside parts that are not incident to $B$, we lose $o(n^{|V(H)|-1})$ copies of $H$.
If $|B|<s-1$, then we can add a vertex to $B$ creating $\Theta(n^{|V(H)|-1})$ copies of $H$, a contradiction. We obtained that $|B|=s-1$. Clearly, for any edge $u_sv_s$ inside parts but outside $B$ we could add independent edges $u_1v_1,\dots,u_{s-1}v_{s-1}$ with $u_i\in B$, to obtain the forbidden configuration. This implies that $G$ is a subgraph of a complete $(s+r-1)$-partite graph with $s$ parts of order 1, completing the proof.
\end{proof}

Let us continue with the proof Theorem \ref{treecompl}, that we restate here for convenience.

\begin{thm*}
Let $F$ be the complete $(r+1)$-partite graph $K_{1,a,\dots,a}$, $H$ be a weakly $F$-Tur\'an-stable graph and $n$ sufficiently large. 

\textbf{(i)} $\ex(n,H,F)=\cN(H,T)$ for some $T\in \cT_0^{(a-1)}(n,r)$.

\textbf{(ii)} If $H$ is a forest, then $\ex(n,H,F)=\cN(H,T)$ for every $T\in \cT_0^{(a-1)}(n,r)$ where the graph embedded into each part has girth at least $|V(H)|$.

\end{thm*}

\begin{proof}
To prove \textbf{(i)}, observe first that by Theorem \ref{thm1}, an $n$-vertex $F$-free graph $G$ with $\cN(H,G)=\ex(n,H,F)$ can be obtained the following way: we embed a $K_{1,a}$-free graph into each part of an $r$-partite graph $T$. We can assume that $T$ is a complete $r$-partite graph, as adding further edges between parts does not create $F$. 

We need to show that the graphs $G_i$ embedded into the parts $A_i$ are $(a-1)$-almost regular. Those graphs $G_i$ do not affect the number of copies of $H$ that contain only edges between parts.
Each other copy of $H$ intersects some part $A_i$ in a subgraph $H_i$ that contains least one edge. Let $p_i$ denote the number of isolated edges in $H_i$ and $q_i$ denote the number of components of $H_i$ of order more than 2. 

Observe that for a fixed $H_i$, there are $O(n^{|V(H)|-p_i-2q_i})$ copies of $H$ intersecting $A_i$ in a copy of $H_i$. Indeed, we pick one vertex from each component inside $A_i$, then there are $O(1)$ ways to pick the other vertices of that component. There are at most $n$ ways to pick each other vertex of $H$. 
This means that for any edge inside $A_i$, there are at least $cn^{|V(H)|-2}$ copies of $H$ containing that edge (and possibly other vertices) for some $c$, and for some $c'$, there are at most $c'n^{|V(H)|-2}$ copies of $H$ containing another subgraph inside $A_i$. 

Assume now that $G_i$ is not an $(a-1)$-almost regular graph. If $|E(G_i)|\le (a-1)|A_i|/2-c'/c$, then we are done, since we lose at least $cn^{|V(H)|-2}c'/c$ copies of $H$. Observe that if $|E(G_i)|> (a-1)|A_i|/2-c'/c$, then at least $|A_i|-2c'/c$ vertices have degree $a-1$. It is easy to see that we can turn $G_i$ to an $(a-1)$-almost regular graph $G_i'$ by adding and removing $O(1)$ edges (removing may be necessary, e.g., if $G_i$ contains two adjacent vertices of degree $a-2$ and $|A_i|-2$ vertices of degree $a-1$). 

We claim that 
$\cN(H_i,G_i)-\cN(H_i,G_i')=O(n^{|V(H_i)|-p_i-q_i-1})$. Let $U$ be the set of vertices with different neighborhood in $G_i$ and $G_i'$, then $|U|=O(1)$. The number of copies of $H_i$ avoiding $U$ is the same in $G_i$ and $G_i'$. The other copies in $G_i'$ can be counted by picking a vertex in $U$ and extending it to a component, $O(1)$ many ways, then picking one vertex in each component, $O(n)$ ways each, then adding adjacent vertices $O(1)$ ways.

This means that for each $H_i$ that is not an isolated edge plus some isolated components, there are at most $c''n^{|V(H)|-3}$ more copies in $G_i$ than in $G_i'$. Therefore, replacing $G_i$ with $G_i'$, we lose $O(n^{|V(H)|-3})$ copies of $H$ and we gain
$\Omega(n^{|V(H)|-2})$ copies, a contradiction.

To prove \textbf{(ii)}, by Theorem \ref{thm1}, we need to maximize the number of copies of $H$ in graphs $G$ which are obtained from a complete $r$-partite graph by placing $K_{1,a}$-free subgraphs into each part. Each copy of $H$ intersects each part in a forest. As observed by Cambie, de Verclos and Kang \cite{cvk}, the largest number of copies of each forest in $m$-vertex $K_{1,a}$-free graphs is in almost $a$-regular graphs with girth at least the order of the forest. Such graphs exist if $m$ is sufficiently large, completing the proof.
\end{proof}

Before proving the rest of the theorems, we need another lemma. Let $H$ be a graph with chromatic number less than $r$ and $uv$ be an edge of $H$ that cuts $H$ into at least two connected components. Let us fix an edge $u'v'$ of $T(n,r)$ and consider the number of copies of $H$ in $T(n,r)$ where $u'v'$ corresponds to $uv$. Note that this number depends on $u'$ and $v'$, let $f(n)$ denote the largest
number obtained this way.

\begin{lemma}\label{lemmanew}
Let $G$ be obtained from $T(n,r)$ by adding and removing $o(n)$ edges at every vertex. Then each edge of $G$ corresponds to $uv$ in $(1+o(1))$ copies of $H$ in $G$. If we add any non-edge to $G$, the resulting edge also corresponds to $uv$ in $(1+o(1))f(n)$ copies of $H$.
\end{lemma}

\begin{proof}
For each edge and non-edge $u'v'$ of $G$, the rest of $G$ has edit distance $o(n^2)$ from the Tur\'an graph, thus has $(1+o(1))\cN(H',T(n,r))$ copies of $H$. The only difference is whether such a copy is extended to $H$ with $u'$ and $v'$. As the parts $A_i$ have roughly the same size, the only difference is whether $u'$ and $v'$ belong to the same part or not. Let $H'$ denote the graph we obtain by deleting $u$ and $v$ from $H$, and let $H''$ be a component of $H'$ such that some of its vertices are connected to, say, $u$. We can embed $H''$ into $G$ $(1+o(1))\cN(H'',T(n,r))$ ways, even after we already embedded some of the other components of $H'$. More importantly, such an embedding is good for us if the neighbors of $u$ in $H'$ are the neighbors of $u'$ in $G$. This happens when they belong to parts of $T(n,r)$ distinct from the part of $u$, with a lower order term of exceptions. Thus we can embed the components one by one, obtaining the same asymptotic each time.
\end{proof}

Now we are ready to prove Theorem \ref{thm3} that we restate here for convenience.

 \begin{thm*}
 Let $r+1=\chi(F)>\chi(H)$, assume that $H$ is $F$-Tur\'an-stable and $F$-Tur\'an-good and $n$ is sufficiently large. Assume that $H$ is a forest
and $\biex(n,F)=O(1)$. Then $\ex(n,H,F)=\cN(H,G)$ for some $G$ 
  with $\ex(n,F)=|E(G)|$.
 \end{thm*}

\begin{proof}
Let $G_0$ be an $n$-vertex $F$-free graph with less than $\ex(n,F)$ edges such that $\cN(H,G_0)=\ex(n,H,F)$, thus we can apply Theorem \ref{thm1}. Therefore, $G_0$ can be obtained from a complete $r$-partite graph $T$ with parts $A_1,\dots,A_r$ by adding $a$ edges inside parts and removing $b$ edges between parts. Moreover, each vertex is connected to $o(n)$ vertices in its part and all but $o(n)$ vertices in the other parts. We also have $a,b=O(1)$. Observe that each part of $T$ has order $(1+o(1))n/r$, otherwise $T$ contains $\Omega(n^{|V(H)|})$ less copies of $H$, and the $O(1)$ extra edges create $O(n^{|V(H)|-2})$ copies of $H$.
Let $G_1$ be an $n$-vertex $F$-free graph with $\ex(n,F)$ edges.

As $K_2$ is $F$-Tur\'an-stable, we have that $G_1$ can be obtained from $T(n,r)$
by adding $a'$ and deleting $b'$ edges. There are no graphs from $\cD(F)$ inside the parts of $T'$, thus $a'=O(1)$. 
Observe that $a-b< a'-b'$ since $G_0$ has less than $\ex(n,F)$ edges. We will apply Lemma \ref{lemmanew} to both $G_0$ and $G_1$.

Let us return to $G_0$. By deleting the $a$ edges inside parts, we delete $(1+o(1))af(n)$ copies of $H$; the number of copies containing more than one such edge is negligible. By adding the $b$ edges to obtain a complete $r$-partite graph, we create $(1+o(1))bf(n)$ copies of $H$. Then we turn the resulting graph to the Tur\'an graph, this does not decrease the number of copies of $H$ because $H$ is $F$-Tur\'an-good. Finally, we turn this graph to $G_1$ by adding $a'$ and removing $b'$ edges. This way altogether the number of copies of $H$ increases by $(1+o(1))(a'-b'+b-a)f(n)$. As we have $0<(a'-b'+b-a)=O(1)$, we obtain that the number of copies of $H$ increases, a contradiction completing the proof.
\end{proof}

We continue with the proof of Theorem \ref{bipa} that we restate below for convenience. As the proof is similar to the proof of Theorem \ref{thm3}, we will present it briefly. One can look at the above proof to find how the arguments can be completed. Recall that $G$ is nice for $F$ if $G$ contains $T(n,r)$ and has $\ex(n,F)$ edges.

\begin{thm*} Let $\chi(F)=3$, $\chi(H)=2$, assume that $H$ is $F$-Tur\'an-stable and $F$-Tur\'an-good, $\biex(n,F)=O(1)$ and $n$ is sufficiently large. If there is a nice graph for $F$, then there is a nice graph $G_0$ for $F$ with $\ex(n,H,F)=\cN(H,G_0)$.
\end{thm*}

\begin{proof}
Let $G$ be an $n$-vertex $F$-free graph $G$ with $\ex(n,H,F)=\cN(H,G)$ and assume that $G$ is not nice. Let us apply Theorem \ref{thm1} to obtain $T$. Let us fix an edge $uv$ of $H$. If $uv$ cuts $H$ into two or more parts, we will apply Lemma \ref{lemmanew} to show that each edge of $G$ corresponds to $uv$ in $(1+o(1))f(n)$ copies of $H$. 
Otherwise $uv$ is contained in a cycle, which must be of even length. Observe that if an even cycle contains an edge of $G$ not in $T$, it must contain another such edge. This implies that an edge of $G$ not in $T$ corresponds to $uv$ in $O(n^{|V(H)|-3})$ copies of $H$. On the other hand, an edge of $T$ corresponds to $uv$ in $\Theta(n^{|V(H)|-2})$ copies of $H$.

We obtained that every edge of $T$ corresponds to $uv$ in more or asymptotically the same number of copies of $H$ as the edges of $G$ outside $T$, which implies the statement. 
\end{proof}

Let us continue with the proof of Theorem \ref{fan}, that we restate here for convenience.

\begin{thm*}

\textbf{(i)} Let $H$ be weakly $Q_r$-Tur\'an-stable. Then $\ex(n,H,Q_r)=\cN(H,G)$ for some $G\in \cG_m$, $m\le r$. In particular, if  $q:=(r-1)/2(r-k+1)$ is not an integer, then $\ex(n,K_k,Q_r)=\cN(K_k,G_{\lceil q\rceil})$. If $q$ is an integer, then either $\ex(n,K_k,Q_r)=\cN(K_k,G_{q})$ or $\ex(n,K_k,Q_r)=\cN(K_k,G_{q+1})$.

\textbf{(ii)} Let $H$ be $B_{r,1}$-Tur\'an-stable. Then $\ex(n,H,B_{r,1})=\cN(H,G')$, for some $G'$ that is obtained from a complete $r$-partite graph by adding an edge into one of the parts.
\end{thm*}

\begin{proof} We start with the upper bound in \textbf{(i)}.
We apply Theorem \ref{thm1}. We obtain that an extremal graph $G$ is obtained by adding at most one edge into each part $B_i$ of an $r$-partite graph $T$. Moreover, every vertex of $T$ is connected to all but $o(n)$ vertices in the other parts. Assume that $u_1v_1$ and $u_2v_2$ are added to parts $B_1$ and $B_2$ of $T$ and assume indirectly that three more edges between these vertices, without loss of generality $u_1u_2,u_1v_2,v_1v_2$ are in $G$. We show that in this case we can embed $F$ into $G$, a contradiction. We embed an edge between $A_r$ and $A_{r+1}$ to $u_1v_1$. For the other edge $xy$ between $A_r$ and $A_{r+1}$, we embed $x$ to a common neighbor of $u_2$ and $v_2$ in $B_1$ and embed $y$ to $u_2$. We embed the single vertex in the first part of $Q_r$ to $v_2$, and embed all the other vertices in $A_r$ and $A_{r+1}$ to neighbors of $v_2$ in $B_1$. Finally, we embed the remaining parts of $Q_r$ by Lemma \ref{greed}.

To show the lower bound in \textbf{(i)}, it is enough to deal with the case $a_2=\dots =a_{r+1}=2$. Assume that $Q_r$ is embedded into $G_m$. Let $v$ denote the vertex in the part of order 1 in $Q_r$. If $v$ is embedded in a part $B_i$ of $T$ together with any other vertex, then they are the endpoints of the extra edge of $G$ inside $B_i$, and the other extra edges of $G$ are not used, since one of the endpoints is not adjacent to the image of $v$. But then $Q_r$ is embedded into $T$ plus one edge, which is impossible. Thus we can assume that $B_i$ contains only $v$.

Observe that if $x\in A_i$ with $i<r$ and $x$ is embedded into $B_i$ of $T$, then at most one other vertex of $Q_r$ is embedded into $B_i$. Indeed, $x$ has only one non-neighbor in $Q_r$, and the same holds for that vertex, thus every set of three vertices in $Q_r$ containing $x$ induces at least two edges. The vertices in parts of order two are embedded into either at least $r-1$ or exactly $r-2$ parts of $T$. 

In the first case, those $r-1$ parts of $T$ contain at most $2r-2$ embedded vertices and the last part contains one vertex, a contradiction. In the second case, those $r-2$ parts contain the images of the parts of order 2, thus $A_r\cup A_{r+1}$ is embedded into one part, a contradiction.

Consider now $\ex(n,K_k,Q_r)$. If $T$ has a part of order not $(1+o(1))n/r$, then clearly $T$ has $\Theta(n^k)$ less copies of $K_k$ then the Tur\'an graph, and the extra edges create $O(n^{k-2})$ copies of $K_k$. Therefore, $G$ is close to $T(n,r)$, in particular, every edge of $G$ inside parts is in $(1+o(1))(n/r)^{k-2}\binom{r-1}{k-2}$ copies of $K_k$. Every non-edge between parts would create $(1+o(1))(n/r)^{k-2}\binom{r-2}{k-2}$ copies of $K_k$ is added to $G$. 

By the first part of the statement in \textbf{(i)}, we know that there are $m$ edges added inside parts and $2\binom{m}{2}$ edges removed between parts. Therefore, the number of copies of $K_k$ is $\cN(K_k,T)+(1+o(1))m(n/r)^{k-2}\binom{r-1}{k-2}-(1+o(1))m(m-1)(n/r)^{k-2}\binom{r-2}{k-2}$.
If $m<q$, then increasing $m$ increases this number, if $m>q$ then increasing $m$ decreases this number. If $m=q$, then it can go either way because of the $o(1)$ term. This shows that $G\in \cG_m$. If there are two parts of $T$ such that $|B_i|>|B_j|+1$, then we move a vertex $u$ from $B_i$ to $B_j$. We pick $u$ that is not incident to an edge of $G$ inside $B_i$. This creates $\Theta(n^{k-2})$ more copies in the complete $r$-partite graph $T$ and destroys $O(n^{k-3})$ copies of $K_k$ that contains an edge of $G$ not in $T$ (since those copies each contain $u$ and one of $m$ edges). Finally, it is easy to see that placing extra edges inside the smaller parts creates more copies of $K_k$.

\smallskip
The lower bound in \textbf{(ii)} is obvious. The upper bound in the case $r=2$ follows from Theorem \ref{bipa}, thus we assume that $r>2$. Assume that there are at least two edges $u_1v_1$ and $u_2v_2$ inside parts. We pick a vertex $v$ that is a common neighbor of these vertices in another part. Then we embed the intersection of the two cliques of $B_{r,1}$ to $v$. We embed two vertices of one of the cliques to $u_1$ and $v_1$ and two vertices of the other clique to $u_2$ and $v_2$. We embed the remaining vertices one by one to the other parts, each time picking a vertex that is in the common neighborhood of the vertices already picked for that clique.
\end{proof}

We finish with the proof of Theorem \ref{stabil} that we restate here for convenience.

\begin{thm*}
\textbf{(i)} Let $\chi(F)=r+1$ and $H$ be a weakly $F$-Tur\'an-stable graph. Assume that $H$ has a unique $r$-coloring and $H'$ is an $r$-chromatic graph obtained by adding edges but no vertices to $H$. Then $H'$ is  weakly $F$-Tur\'an-stable. 

\textbf{(ii)} Let $\chi(F)=3$ and $H$ be a weakly $F$-Tur\'an-stable graph. Let us assume that $H$ contains the $b(F)$-blow-up of $P_{\og(F)-1}$, where $u_1,\dots,u_{b(F)}$ replace the first vertex and $v_1,\dots,v_{b(F)}$ replace the last vertex of the path. Let $H'$ be the graph obtained by adding vertices $w_1,\dots,w_s,x_1,\dots x_t$ and edges $u_iw_j$, $v_ix_\ell$ for $i\le b(F)$, $j\le s$ and $\ell\le t$. Assume that $\binom{t-s}{2}\le s\le t$. Then $H'$ is $F$-Tur\'an-stable.
\end{thm*}

\begin{proof}
To prove \textbf{(i)}, assume that there are $p$ ways to obtain $H'$ from $H$ by adding edges, and $H'$ contains $q$ copies of $H$. Observe that if a copy of $H$ is in a complete $r$-partite graph $G_0$, then all the $p$ ways to obtain $H'$ create a subgraph of $G_0$. 

We claim that $\ex(n,H,F)=q\ex(n,H',F)/p-o(n^{|V(H)|})$. Let $G'$ be an $n$-vertex $F$-free graph with $\ex(n,H,F)$ copies of $H$, then $G'$ contains at least $p\ex(n,H,F)/q$ copies of $H'$. Let $T$ be the complete $r$-partite graph obtained by adding and deleting $o(n^2)$ edges from $G_0$, then $G_0$ contains $\cN(H,G')-o(n^{|V(H)|})$ copies of $H$. Since $G_0$ contains at least $p\cN(H,G_0)/q$ copies of $H'$, we have that $\ex(n,H,F)=\cN(H,G')=\cN(H,G_0)-o(n^{|V(H)|})\le q\cN(H',G_0)/p-o(n^{|V(H)|})\le q\ex(n,H',F)/p-o(n^{|V(H)|})$.

Let $G$ be an $n$-vertex $F$-free graph with $\ex(n,H',F)-o(n^{|V(H)|})$ copies of $H'$. Clearly $G$ contains at least $q\ex(n,H',F)/p-o(n^{|V(H)|})=\ex(n,H,F)-o(n^{|V(H)|})$ copies of $H$. Therefore, by the weak $F$-Tur\'an-stability of $H$, $G$ can be obtained from a complete $r$-partite graph by adding and removing $o(n^2)$ edges, completing the proof.

To prove \textbf{(ii)}, let $G$ be an $n$-vertex $F$-free graph. We first embed $H$ and then the additional vertices. Let $Q$ be a copy of $H$ and observe that there are at most $b(F)-1$ common neighbors of $u_1,\dots,u_{b(F)},v_1,\dots,v_{b(F)}$ in $G$. Let $A$ denote the set of vertices that are common neighbors of $u_1,\dots,u_{b(F)}$ and $B$ denote the set of vertices that are common neighbors of $v_1,\dots,v_{b(F)}$ in $G$, then $|A\cap B|\le b(F)-1$. We need to pick $s$ vertices from $A$ and $t$ vertices from $B$. 

Observe that the number of ways to do this while avoiding the vertices in $A\cap B$ is the same as the number of ways to pick $K_{s,t}$ from $K_{|A\setminus B|,|B\setminus A|}$. This is asymptotically the largest if $|A|=|B|$ by a result of Brown and Sidorenko \cite{brosid}. Note that they did this optimization in a slightly different context. Observe that $A\cap B$ is negligible, thus we have at most $(1+o(1))\cN(K_{s,t},T(n-|V(H)|,2))$ ways to extend a copy of $H$ to $H'$. Therefore, the number of copies of $H'$ in $G$ is at most the number of copies of $H$ times $(1+o(1))\cN(K_{s,t},T(n-|V(H)|,2))$, divided by some fixed number $q$ of automorphisms of $H'$. 

Observe that removing $|V(H)|$ vertices of $T(n,2)$, the resulting graph has edit distance $O(n)$ from $T(n-|V(H)|,2)$. Therefore, $\cN(H',T(n,2))=(1+o(1))\cN(H,T(n,2))\cN(K_{s,t},T(n-|V(H)|,2))/q$. Assume now that $G$ has $$\ex(n,H',F)-o(n^{|V(H')|})\ge (1+o(1))\cN(H,T(n,2))\cN(K_{s,t},T(n-|V(H)|,2))/q$$ copies of $H'$. Then $G$ has to contain $\cN(H,T(n,2))-o(n^{|V(H)|})=\ex(n,H,F)-o(n^{|V(H)|})$ copies of $H$, thus $G$ has edit distance $o(n^2)$ from $T(n,2)$ by the $F$-Tur\'an-stable property of $H$, completing the proof.
\end{proof}

\bigskip

\textbf{Funding}: Research supported by the National Research, Development and Innovation Office - NKFIH under the grants KH 130371, SNN 129364, FK 132060, and KKP-133819.

\end{document}